\documentclass[12pt]{amsart}
\usepackage[utf8]{inputenc}
\usepackage{amsmath,amssymb,amsthm,amsfonts,mathrsfs}
\usepackage[pdftex,plainpages=false,hypertexnames=false,pdfpagelabels]{hyperref}
\usepackage{xcolor}
\usepackage{enumitem}
\definecolor{dark-red}{rgb}{0.7,0.25,0.25}
\definecolor{dark-blue}{rgb}{0.15,0.15,0.55}
\definecolor{medium-blue}{rgb}{0,0,.8}
\definecolor{DarkGreen}{RGB}{0,150,0}
\definecolor{rho}{named}{red}
\hypersetup{
   colorlinks, linkcolor={purple},
   citecolor={medium-blue}, urlcolor={medium-blue}
}

\newcommand{\cB}{\mathcal{B}}
\newcommand{\D}{\mathbb{D}}

\newcommand{\cJ}{\mathcal{J}}

\newcommand{\N}{\mathbb{N}}
\newcommand{\R}{\mathbb{R}}
\newcommand{\C}{\mathbb{C}}

\renewcommand{\Re}{\operatorname{Re}}
\newcommand{\Span}{\operatorname{span}}
\newcommand{\End}{\operatorname{End}}
\newcommand{\rank}{\operatorname{rank}}
\newcommand{\length}{\operatorname{length}}
\newcommand{\Hilb}{\operatorname{Hilb}}

\newcommand{\isoarrow}{\overset{\sim}{\longrightarrow}}

\newcommand{\abs}[1]{\left|#1\right|}
\newcommand{\norm}[1]{\left\|#1\right\|}

\newtheorem{thmalpha}{Theorem}

\newtheorem{Theorem}{Theorem}[section]

\newtheorem*{Theorem*}{Theorem}
\newtheorem{Lemma}[Theorem]{Lemma}
\newtheorem{Proposition}[Theorem]{Proposition}
\newtheorem{Corollary}[Theorem]{Corollary}
\theoremstyle{definition}
\newtheorem{Remark}[Theorem]{Remark}

\newtheorem*{Definition*}{Definition}

\newtheorem{Question}[Theorem]{Question}
\numberwithin{equation}{section}
\numberwithin{figure}{section}

\begin{document}

\title[]{Singular values of weighted composition operators and second quantization}


\author{Mihai Putinar}
\address{Dept. of Mathematics\\
    University of California\\
    Santa Barbara, CA 93106-6105\\
    USA}
\address{
   Dept. of Mathematics and Statistics\\
   Newcastle University\\
  NE1 7RU, United Kingdom
}
\email{mputinar@math.ucsb.edu}

\author{James E. Tener}
\address{Dept. of Mathematics\\
    University of California\\
    Santa Barbara, CA 93106-6105\\
    USA}
\email{jtener@math.ucsb.edu}


\begin{abstract}
We study a semigroup of weighted composition operators on the Hardy space of the disk $H^2(\D)$, and more generally on the Hardy space $H^2(U)$ attached to a simply connected domain $U$ with smooth boundary.
Motivated by conformal field theory, we establish bounds on the singular values (approximation numbers) of these weighted composition operators. 
As a byproduct we obtain estimates on the singular values of the restriction operator (embedding operator) $H^2(V) \to H^2(U)$ when $U \subset V$ and the boundary of $U$ touches that of $V$.
Moreover, using the connection between the weighted composition operators and restriction operators, we show that these operators exhibit an analog of the Fisher-Micchelli phenomenon for non-compact operators.
\end{abstract}
\maketitle

\section{Introduction}

There is a fundamental link between the free fermion conformal field theory and complex function theory.
Indeed, in Graeme Segal's landmark paper \cite{SegalDef} he described how to construct the free fermion CFT as the second quantization of the Hardy spaces of planar domains with smooth boundary (or, more generally, Riemann surfaces with smooth boundary).
The properties of the second quantized operators are derived from analytic properties of operators on these Hardy spaces.
For example, in \cite{Ten16a} the boundedness and compactness of second quantized operators are derived from a careful study of the Cauchy transform.
Moreover, in \cite{Ten16b} it was shown that a necessary ingredient of the dictionary between geometric and algebraic formulations of the free fermion conformal field theory is the boundedness of the norms of the exterior powers of certain weighted composition operators $W_{\varphi}$, which we define below, on the Hardy space $H^2(\D)$.
By Theorem \ref{thmLambdaBoundedness}, the boundedness of these exterior power operators $\Lambda(W_\varphi)$ is characterized by the decay rate of the singular values (sometimes called \emph{approximation numbers} in the non-compact case) $s_n(W_\varphi)$ of the prequantized operator $W_\varphi$.

Inspired by these results, we perform a general study of the spectral properties of these weighted composition operators, and relatedly the spectral analysis of the restriction operator between two Jordan domains with smooth boundary.
We also generalize a celebrated theorem of Fisher-Micchelli \cite{FisherMicchelli80} concerning n-widths in Hardy space to the limiting case of a non-compact restriction operator. 
It is surprising, and gratifying, that some qualitative features (decay and simplicity of eigenvectors plus the oscillation of eigenvalues) persist in this new situation.

The following theorem summarizes our main results, which we prove in the body of the paper in Theorem \ref{thmSingularValues} and Theorem \ref{thmFisherMicchelli}.
\begin{thmalpha}\label{thmIntroSummary}
Let $U$ and $V$ be Jordan domains with $C^\infty$ boundary, and let $\varphi:U \to V$ be a univalent map such that $\varphi(U)$ is also a Jordan domain with $C^\infty$ boundary.
For either choice of holomorphic square root of $\varphi^\prime$, define the weighted composition operator 
$$
W_\varphi:H^2(V) \to H^2(U)
$$
by $(W_\varphi f)(z) = \varphi^\prime(z)^{1/2} f(\varphi(z))$.
Then the following hold:
\begin{enumerate}
\item The essential norm of $W_{\varphi}$ is given by
$$
\norm{W_\varphi}_e = \left\{ 
\begin{array}{cl}
1, & \mbox{if } \, \partial \varphi(U) \cap \partial V \ne \emptyset,\\
0, & \mbox{if } \,\partial \varphi(U) \cap \partial V = \emptyset.
\end{array}
\right.
$$
\item The singular values of $W_\varphi$ satisfy 
$
s_n(W_\varphi) \le 1 + \frac{K}{n}
$ for some positive constant $K$.
\item If $\varphi(U) \ne V$ and $\lambda$ is an eigenvalue of $\abs{W_\varphi}$, then $\lambda$ has multiplicity one.
\item If $\lambda_0 > \lambda_1 > \cdots > \norm{W_\varphi}_e$ are the eigenvalues of $\abs{W_{\varphi}}$ exceeding $\norm{W_{\varphi}}_e$, then the eigenfunction corresponding to the eigenvalue $\lambda_n$ has at most $n$ zeroes in $V$.
\end{enumerate}
In particular, if $U \subseteq V$ and $\iota:U \to V$ is the inclusion map, then the restriction operator $R:H^2(V) \to H^2(U)$ is given by $R = W_\iota$ for the appropriate choice of square root, and thus the above results apply to $R$.
\end{thmalpha}

In the case where $(\phi_t)_{t \ge 0}$ is a sufficiently nice semigroup of maps as above with an interior fixed point, then we can show the stronger result that $\prod \max(1,s_n(W_{\varphi_t})) < \infty$ for all $t \ge 0$ as a consequence of \cite{Ten16b}.
The proof, however, is highly indirect, and requires relating the exterior power semigroup $\Lambda(W_{\varphi_t})$ to a semigroup of operators generated by a smeared Virasoro field and applying the `quantum energy inequality' of \cite{FewsterHollands05} to obtain a one-sided bound the spectrum of the real part of the generator.
It would be very interesting to have a direct proof of this fact, in the spirit of the analysis performed to obtain Theorem \ref{thmIntroSummary}, as well as a generalization of the result which applies to maps $\varphi$ not lying in semigroups.

Intriguingly, the link which we explore between singular values of weighted composition operators and two-dimensional conformal field theory is not the only such connection.
Very recently, singular values of these operators were studied (with emphasis on \emph{compact} operators) in \cite{LLQR16}, in the context of modular nuclearity for Borchers triples, a very different application than the one explored here.
The simultaneous appearance of two distinct applications of singular values of weighted composition operators to conformal field theory reveals their importance, and suggests further work aimed at understanding the relationship between these connections.

The paper is organized as follows.
In Section \ref{secSingularValues}, we introduce the weighted composition operators $W_\varphi$ and establish estimates for their singular values.
In Section \ref{secSecondQuantization} we discuss the relation between singular values of $W_\varphi$ and the boundedness of $\Lambda(W_\varphi)$ which is required for applications in conformal field theory.
In Section \ref{secRestriction} we exploit the relationship between the weighted composition operators $W_\varphi$ and restriction operators to establish a Fisher-Micchelli-type theorem for non-compact $W_\varphi$.

\subsection*{Acknowledgements}
James Tener is grateful to the Max Planck Institute for Mathematics, Bonn, for its hospitality and support.

\section{Singular values of a family of weighted composition operators}\label{secSingularValues}

Let $U \subset \C$ be a bounded, simply connected domain.
For a prescribed point $z_0 \in U$, one defines the Hardy space $H^2(U, \omega_{z_0})$ in terms of the harmonic measure $\omega_{z_0}$ with pole at $z_0$. 
If $(U, z_0)$ and $(V,w_0)$ are a pair of such domains, and $\varphi:U \to V$ is a holomorphic map with $\varphi(z_0) = w_0$, then it is natural to consider the composition operator $C_\varphi:H^2(V,\omega_{w_0}) \to H^2(U, \omega_{z_0})$ given by $C_\varphi f = f \circ \varphi$.
One has $\norm{C_\varphi} \le 1$, and $C_\varphi$ is unitary when $\varphi$ is invertible.

Alternatively, if $U$ has $C^\infty$ boundary boundary there is another notion of Hardy space, which we denote $H^2(U)$, obtained as the closure in $L^2(\partial U, ds)$ of the subspace of
smooth functions on $\overline{U}$ which are holomorphic in $U$.
Here $ds$ is arclength measure on $\partial U$.
This is equivalent to working with holomorphic functions on $U$ with almost everywhere non-tangential $L^2$ boundary values (see \cite[\S 5]{Bell92}).
As topological vector spaces, $H^2(U)$ and $H^2(U, \omega_{z_0})$ coincide, but it is an interesting and subtle problem to study their differences as Hilbert spaces.

The second version of the Hardy space is most natural to consider when we regard
$U$ endowed with a spin structure (which is unique up to isomorphism).
In this case, a map $H^2(V) \to H^2(U)$ should be obtained from a morphism $U \to V$ of \emph{spin structures}, not just a map of the underlying domains.
When $\varphi:U \to V$ is a locally injective holomorphic map,  a morphism of spin structures is obtained by choosing a holomorphic square root of the derivative $\varphi^\prime$.
Given such a $\varphi: U \to V$ equipped with a choice of $(\varphi^\prime)^{1/2}$, we define the weighted composition operator
$W_{\varphi}:H^2(V) \to H^2(U)$ by $(W_{\varphi}f)(z) = \varphi^\prime(z)^{1/2} f(\varphi(z))$.
Note that $W_{\varphi_1} W_{\varphi_2} = W_{\varphi_2 \circ \varphi_1}$ when the square roots are chosen compatibly.
Analogous to the case of composition operators and harmonic measure, $W_\varphi$ is unitary when $\varphi$ is invertible.

Let $\cJ^\infty$ denote the set of bounded, simply connected Jordan domains $U \subset \C$ with $C^\infty$ boundary, and for $U,V \in \cJ^\infty$ let $\cJ^\infty(U, V)$ denote the set of univalent maps $\varphi:U \to V$, equipped with a choice of square root of $\varphi^\prime$, such that $\varphi(U) \in \cJ^\infty$.
For $\varphi \in \cJ^\infty(U,V)$, the weighted composition operator $W_\varphi:H^2(V) \to H^2(U)$ is bounded, as $\varphi$ extends to a diffeomorphism $\partial U \isoarrow \partial \varphi(U)$, and in particular $\varphi^\prime$ extends continuously to $\overline{U}$.
Unlike the situation with harmonic measure, however, $W_\varphi$ is not a contraction in general.
Indeed,
$$
\norm{W_{\varphi}} \ge \frac{\norm{W_{\varphi} 1}_{H^2(U)}}{\norm{1}_{H^2(V)}} = \left(\frac{\length(\partial \varphi(U))}{\length(\partial V)}\right)^{\frac12},
$$
and thus $\norm{W_{\varphi}}$ may be arbitrarily large.

On the other hand, we will show that $W_\varphi$ is an \emph{essential} contraction.
Recall that the \emph{singular values} (often called \emph{approximation numbers} in the case of non-compact operators) of $T \in \cB(H,K)$ are defined by
$$
s_n(T) = \inf \{ \norm{T - F} \, : \, F \in \cB(H,K), \, \rank(F) < n\},
$$
and the \emph{essential norm} of $T$ is given by $\norm{T}_e := \lim_{n \to \infty} s_n(T)$. 
The classical definition of essential norm
$$
\norm{T}_e = \inf \{ \norm{T - L} \, : \, L \in \cB(H,K), \, L \ {\rm compact} \}
$$
is equivalent to the above limit of singular values, and it is obvious that it 
only depends on the modulus $|T| = \sqrt{T^\ast T}$ of $T$.

\begin{Theorem}\label{thmSingularValues}
Let $U,V \in \cJ^\infty$ and let $\varphi \in \cJ^\infty(U,V)$.
Then
$$
\norm{W_\varphi}_e = \left\{ 
\begin{array}{cl}
1, & \mbox{if } \, \partial \varphi(U) \cap \partial V \ne \emptyset,\\
0, & \mbox{if } \,\partial \varphi(U) \cap \partial V = \emptyset.
\end{array}
\right.
$$
Moreover, there exists a positive constant $K_\varphi $ with the property
$$
s_n(W_\varphi) \le 1 + \frac{K_\varphi}{n},
$$
for all $n \geq 1$.
\end{Theorem}
\begin{proof}
By composing $\varphi$ with Riemann maps for $U$ and $V$, which corresponds to multiplying $W_\varphi$ by unitary operators, we may assume without loss of generality that $U = V = \D$. 
Moreover, we may choose the Riemann maps so that $\varphi(0) = 0$.
Following convention, we will use \emph{normalized} arclength measure on the unit circle $S^1$ to define the inner product on $H^2(\D)$, but since we are rescaling the norm on both the domain and codomain the conclusions of the theorem are not affected.
If $\norm{\varphi}_\infty < 1$, so that $\partial \varphi(\D) \cap \partial \D = \emptyset$, then $W_\varphi$ is trace class and the conclusion follows.
Thus we assume that $\norm{\varphi}_\infty = 1$.

The following argument is inspired by the computation of $\norm{C_\varphi}_e$ in \cite{Shapiro87}.
Let $D$ be the unbounded operator on $H^2(\D)$ given by $Df = f^\prime$, with domain consisting of functions $f \in H^2(\D)$ with the property that $f^{(n)}$ extends continuously to $\overline{\D}$ for all $n \ge 0$.
Let $A \in \cB(H^2(\D))$ be the compact operator given by $Az^n = \frac{1}{n+1} z^{n+1}$, and observe that $s_n(A) = \frac{1}{n}$.
If $\psi \in H^\infty(\D)$, we write $M_\psi$ for the multiplication operator by $\psi$ on $\cB(H^2(\D))$.

For $f \in\C[z]$, we have
\begin{align*}
W_\varphi f &= (W_\varphi f)(0) + ADW_\varphi f\\
 &= \varphi^\prime(0)^{1/2}f(0) + AM_{(\varphi^\prime)^{3/2}} C_\varphi Df + A M_\psi C_\varphi f 
\end{align*}
where $\psi = ((\varphi^\prime)^{1/2})^\prime = \frac{\varphi^{\prime\prime}}{2(\varphi^\prime)^{1/2}}$.
By the smooth Riemann mapping theorem (see \cite[Thm. 12.1]{Bell92}), $\varphi^{\prime\prime}$ and $\frac{1}{\varphi^\prime}$ lie in $H^\infty(\D)$.
Hence 
$$
s_n(A M_\psi C_\varphi) \le \frac{\norm{\psi}_\infty}{n}.
$$

To obtain the desired estimate on $s_n(W_\varphi)$, it suffices to show that the map $T$ defined on $f \in \C[z]$ by
$$
Tf = \varphi^\prime(0)^{1/2} f(0) + AM_{(\varphi^\prime)^{3/2}} C_\varphi Df
$$
extends to a bounded operator with norm at most $1$.
To do this, we use the Littlewood-Paley formula for the $H^2(\D)$ norm (see \cite[\S 4.5(a)]{Shapiro87})
$$
\norm{f}_{H^2(\D)}^2 = \abs{f(0)}^2 + \int_\D \abs{f^\prime(z)}^2 \log(1/\abs{z}^2) \, d\lambda(z)
$$
where $\lambda$ is Lebesgue measure on $\C$, normalized so that $\lambda(\D)=1$.

Hence, for $f \in \C[z]$, we have
\begin{align}
\norm{Tf}^2 & = \abs{\varphi^\prime(0)}\abs{f(0)}^2 + \int_\D \abs{\varphi^\prime(z)}^3 \abs{f^\prime(\varphi(z))}^2 \log(1/\abs{z}^2) \, d\lambda(z) \nonumber\\
&= \abs{\varphi^\prime(0)}\abs{f(0)}^2 + \nonumber \\
&\quad \int_\D \left(\frac{\abs{\varphi^\prime(z)}\log\abs{z}}{\log\abs{\varphi(z)}}\right) \abs{\varphi^\prime(z)}^2 \abs{f^\prime(\varphi(z))}^2 \log(1/\abs{\varphi(z)}^2) \, d\lambda(z). \label{eqnTEstimate}
\end{align}
By the Schwarz-Pick lemma,
for $z \ne 0$ we have
\begin{align*}
\frac{\abs{\varphi^\prime(z)}\log\abs{z}}{\log\abs{\varphi(z)}} & = \frac{\abs{\varphi^\prime(z)}(1-\abs{z}^2)}{1-\abs{\varphi(z)}^2}\frac{\log\abs{z}}{1-\abs{z}^2}\frac{1-\abs{\varphi(z)}^2}{\log\abs{\varphi(z)}}\\
& \le \frac{\log\abs{z}}{1-\abs{z}^2}\frac{1-\abs{\varphi(z)}^2}{\log\abs{\varphi(z)}}\\
&= \frac{u(\abs{z})}{u(\abs{\varphi(z)})}
\end{align*}
where $u:(0,1) \to \R_{> 0}$ is given by $u(t) = -\log t/(1-t^2)$.
One verifies that $u$ is decreasing, and since $\abs{\varphi(z)} \le \abs{z}$ by the Schwarz lemma, we have
\begin{equation}\label{eqnPhiEstimate}
\frac{\abs{\varphi^\prime(z)}\log\abs{z}}{\log\abs{\varphi(z)}} \le 1.
\end{equation}

Combining \eqref{eqnPhiEstimate} with \eqref{eqnTEstimate}, and the fact that $\abs{\varphi^\prime(0)} \le 1$ by the Schwarz lemma, we get
\begin{align*}
\norm{Tf}^2 &\le \abs{f(0)}^2 + \int_\D  \abs{\varphi^\prime(z)}^2 \abs{f^\prime(\varphi(z))}^2 \log(1/\abs{\varphi(z)}^2) \, d \lambda(z)\\
&= \abs{f(0)}^2 + \int_{\varphi(\D)} \abs{f^\prime(z)}^2 \log(1/\abs{z}^2) \, d\lambda(z)\\
&\le \norm{f}^2.
\end{align*}
Since $W_\varphi = T + AM_{\psi}C_\varphi$ and $s_n(AM_\psi C_\varphi) \le \frac{K_\varphi}{n}$, we have established the desired estimate for $s_n(W_\varphi)$. Moreover it follows that $\norm{W_\varphi}_e \le 1$.

It remains to establish a lower bound for $\norm{W_\varphi}_e$, which we will do as in \cite{Shapiro87}.
Since $\norm{\varphi}_\infty = 1$, we may choose a sequence $w_n \in \D$ with $\abs{\varphi(w_n)} \to 1$.
By the smooth Riemann mapping theorem, $\varphi$ and all of its derivative extend continuously to the closed unit disk $\overline{\D}$.
By taking a subsequence, we may assume that $w_n$ converges to a point on the unit circle, and thus $\varphi(w_n)$ does as well.

For $w \in \D$, let $k_w$ be the corresponding normalized reproducing kernel function, namely
$$
k_w(z) = \frac{(1-\abs{w}^2)^{1/2}}{1-\overline{w}z}.
$$
Since the sequence of unit norm functions $k_{w_n}$ converges weakly to $0$, to show that $\norm{W_\varphi}_e \ge 1$ it suffices to show that $\lim_{n \to \infty} \norm{W_\varphi^* k_{w_n}} = 1$.
This is easily achieved, as
$$
W_\varphi^* k_w = C_\varphi^* M_{(\varphi^\prime)^{1/2}}^* k_w =  \left(\frac{\overline{\varphi^\prime(w)}(1-\abs{w}^2)}{1-\abs{\varphi(w)}^2}\right)^{1/2} k_{\varphi(w)}.
$$
Hence $\norm{W_\varphi^* k_w} \le 1$ by the Schwarz-Pick lemma. 
Moreover, by the Julia-Carath\'eodory theorem (see \cite[\S4.2]{Shapiro93}) we have
$$
\norm{W_\varphi^* k_{w_n}}^2 = \frac{\abs{\varphi^\prime(w_n)}(1-\abs{w_n})^2}{1-\abs{\varphi(w_n)}^2} \to 1.
$$
\end{proof}

Theorem \ref{thmSingularValues} shows that $\norm{W_\varphi}_e$ captures basic geometric information about the inclusion $\varphi(U) \subseteq V$ when $\varphi \in \cJ^\infty(U,V)$, namely whether or not the boundaries of $\varphi(U)$ and $V$ intersect.
It is desirable to have a geometric understanding of the meaning of the singular values $s_n(W_\varphi)$, or more generally of $\abs{W_\varphi}$.
In Section \ref{secRestriction}, we will explore further the connection between $\abs{W_\varphi}$ and the inclusion $\varphi(U) \subseteq V$ via the restriction operator $R: H^2(V) \to H^2(\varphi(U))$.

\begin{Remark}
The singular value estimates obtained in Theorem \ref{thmSingularValues} are much more precise than one can generally compute for non-compact weighted composition operators.
For example, the techniques employed in \cite[Thm. 2.2]{LLQR16}, which apply to a broader class of weighted composition operators, cannot easily recover even $\norm{W_\varphi}_e$.
This suggests that the family $W_\varphi$ form a particularly interesting class of examples, deserving of further study.
\end{Remark}

\section{Second quantization}\label{secSecondQuantization}

One motivation for studying the operators $W_\varphi$ comes from conformal field theory.
Let $H$ be a Hilbert space, and let $\Lambda H = \bigoplus_{n=0}^\infty \Lambda^n H$ be the antisymmetric Fock space of $H$.
Given $T \in \cB(H,K)$, the map $\Lambda(T): \Lambda H \to \Lambda K$ given by
$$
T(v_1 \wedge \cdots \wedge v_n) = Tv_1 \wedge \cdots \wedge Tv_n
$$
is only densely defined in general, although $\norm{\Lambda(T)} \le 1$ when $\norm{T} \le 1$.

When $\varphi \in \cJ^\infty(\D,\D)$ and $\norm{\varphi}_\infty < 1$,  $\Lambda(W_\varphi)$ is closely related to the second quantization of the Hardy space of the annulus $\overline{\D} \setminus \varphi(\D)$.
In \cite{Ten16b}, the second author considered the case when $\norm{\varphi}_\infty = 1$, and showed that the boundedness of $\Lambda(W_\varphi)$ was important for building a dictionary between the geometric and operator algebraic formulations of the free fermion conformal field theory.
However, boundedness of $\Lambda(W_\varphi)$ was only established when $\varphi$ could be fit into a nice one parameter semigroup $(\varphi_t)_{t \ge 0} \in \cJ^\infty(\D,\D)$ with interior Denjoy-Wolff point (see also Corollary \ref{corSemigroups}).
The boundedness of $\Lambda(W_\varphi)$ is related to the singular values studied in Section \ref{secSingularValues} by the following result.

\begin{Theorem}\label{thmLambdaBoundedness}
Let $H$ and $K$ be Hilbert spaces, and let $T \in \cB(H, K)$.
Then the following are equivalent:
\begin{enumerate}[label=(\arabic*)]
\item $\Lambda(T)$ is bounded. \label{itmBoundedness}
\item $T$ can be written $T = A + X$, with $A,X \in \cB(H,K)$, $\norm{A} \le 1$, and $X$ trace class. \label{itmDecomposition}
\item $\prod_{n=1}^\infty \max(1,s_n(T)) < \infty$. \label{itmNormProduct}
\end{enumerate}
If these conditions hold, then $\norm{\Lambda(T)}$ is given by the product in condition \ref{itmNormProduct}.
\end{Theorem}
The conditions of Theorem \ref{thmLambdaBoundedness} are clearly satisfied when $\norm{T}_e < 1$, and can never be satisfied when $\norm{T}_e > 1$.
When $\norm{T}_e = 1$, the last condition reduces to $\prod_{n=1}^\infty s_n(T) < \infty$, or equivalently $\sum (s_n(T) - 1) < \infty$.
It is possible that Theorem \ref{thmLambdaBoundedness} is well-known, but we were unable to find a reference, and so we give a proof at the end of this section.
We note that the equivalence of conditions \ref{itmBoundedness} and \ref{itmDecomposition} is implicit in \cite[\S5]{Ten16b}.

\begin{Corollary}\label{corSemigroups}
Let $U \in \cJ^\infty$, and let $(\varphi_t)_{t \ge 0}$ be a continuous semigroup of maps with $\varphi_t \in \cJ^\infty(U,U)$.
Suppose that there exists a $z_0 \in U$ such that $\varphi_t(z_0) = z_0$ for all $t \ge 0$.
Moreover, suppose that the Koenigs function $\sigma$ of $\varphi_t$ maps $U$ onto a Jordan domain with $C^\infty$ boundary.
Then $\prod_{n=1}^\infty \max(1,s_n(W_{\varphi_t})) < \infty$ for all $t \ge 0$.
\end{Corollary}
\begin{proof}
The semigroup $W_{\varphi_t}$ is unitarily equivalent to a semigroup of self-maps of $\D$ which fix $z_0 = 0$.
It was shown in the proof of \cite[Thm. 3.21]{Ten16b} that $\Lambda(W_{\varphi_t})$ is bounded in this case, and the corollary then follows immediately from Theorem \ref{thmLambdaBoundedness}.
\end{proof}

We expect that the bounds from Theorem \ref{thmSingularValues} can be improved to $\prod s_n(W_\varphi) < \infty$ for arbitrary $\varphi \in \cJ^\infty(U,V)$.
It would follow that there is a contravariant functor
$
\cJ^\infty \to \Hilb
$
sending $U$ to $\Lambda H^2(U)$ and $\varphi$ to $\Lambda(W_\varphi)$.
The endomorphisms of $\D$ in $\cJ^\infty$ are related to an extension of Segal's semigroup of annuli to `degenerate annuli,' where the incoming boundary and outgoing boundary of the annuli are allowed to overlap, and the functor $\varphi \mapsto \Lambda(W_\varphi)$ should be thought of as a representation of $\End(\D)$.

We now turn to proving Theorem \ref{thmLambdaBoundedness}.
The key step in the proof is the following useful observation.

\begin{Lemma}\label{lemNormExteriorPower}
Let $H$ and $K$ be Hilbert spaces, and let $T \in \cB(H, K)$.
For $n\ge 0$, let $\Lambda^n(T) \in \cB(\Lambda^n H, \Lambda^n K)$ be the map $\Lambda^n(T)(v_1 \wedge \cdots \wedge v_n) = Tv_1 \wedge \cdots \wedge T v_n$.
Then $\norm{\Lambda^n(T)} = \prod_{j=1}^n s_j(T)$.
\end{Lemma}
\begin{proof}
We have $\abs{\Lambda^n(T)} = \Lambda^n (\abs{T})$ and $s_j(T) = s_j(\abs{T})$, so we may assume without loss of generality that $H = K$ and $T \ge 0$. 
If $n > \dim H$ then the conclusion is clear, so we assume that $\dim H \ge n$.
As $T$ will remain fixed, we will write $s_j$ instead of $s_j(T)$.

We first consider the case when $\norm{T} = \norm{T}_e$.
Since $\Lambda^n(T)$ is the restriction of $T^{\otimes n}$ to an invariant subspace, we have 
$$
\norm{\Lambda^n(T)} \le \norm{T}^n = \prod_{j=1}^n s_j.
$$
On the other hand, since the spectral projection associated to any neighborhood of the essential norm has infinite rank, for any $\epsilon > 0$ there exists an orthonormal family of vectors $v^\epsilon_1, \ldots, v^\epsilon_n$ such that the collection $Tv^\epsilon_j$ is also orthogonal, and $\|Tv^\epsilon_j\| > \norm{T} - \epsilon$.
Hence 
$$
\norm{\Lambda^n(T)} \ge \sup_\epsilon \norm{Tv^\epsilon_1 \wedge \cdots \wedge Tv^\epsilon_n} = \norm{T}^n,
$$
which completes the proof when $\norm{T} = \norm{T}_e$.

Next suppose $s_n > \norm{T}_e$.
Then each $s_n$ is an eigenvalue of $T$, and we may choose an orthornormal family of eigenvectors $v_1, \ldots, v_n$ so that $Tv_j = s_j v_j$.
Let $K = \Span \{v_1, \ldots, v_n\}^\perp$.
There is a natural unitary isomorphism 
$$
\Lambda^n H \cong \bigoplus_{k_0, \ldots, k_n} \Lambda^{k_0} K \otimes \Lambda^{k_1} \C v_1 \otimes \cdots \otimes \Lambda^{k_n} \C v_n
$$
where the direct sum is indexed by nonnegative integers $k_j$ which satisfy $k_0 + \cdots + k_n = n$ and $k_j \le 1$ for $j \ge 1$.
This unitary identifies $\Lambda^n(T)$ with
$$
\bigoplus_{k_0, \ldots, k_n} (s_1^{k_1} \cdots s_n^{k_n}) \Lambda^{k_0} (T|_K) \otimes 1 \otimes \cdots \otimes 1.
$$
We have $s_1 \ge \cdots \ge s_n \ge \norm{T|_K}$, and thus 
$$
\norm{\Lambda^n(T)} = \max_{k_1,\ldots,k_n} (s_1^{k_1} \cdots s_n^{k_n}) \norm{\Lambda^{k_0} T|_K} = s_1 \cdots s_n.
$$

Finally, consider when $\norm{T} \ne \norm{T}_e$ but $s_n = \norm{T}_e$.
Let $m$ be such that $s_m > \norm{T}_e$ but $s_{m+1} = \norm{T}_e$.
Then proceeding as above, choose an orthonormal set of eigenvectors $v_1, \ldots, v_m$ satisfying $T v_j = s_j v_j$, and let $K = \Span \{v_1, \ldots, v_m\}^\perp$.
Then $\Lambda^n(T)$ is unitarily equivalent to
$$
\bigoplus_{k_0, \ldots, k_m} (s_1^{k_1} \cdots s_m^{k_m}) \Lambda^{k_0} (T|_K) \otimes 1 \otimes \cdots \otimes 1
$$
where the direct sum is indexed by nonnegative integers $k_j$ which satisfy $k_0 + \cdots + k_m = n$ and $k_j \le 1$ when $j \ge 1$.
Note that $\norm{T|_K} = \norm{T|_K}_e = \norm{T}_e$, so that $\norm{\Lambda^{k_0}(T|_K)} = \norm{T}_e^{k_0}$ by the first case we considered.
Hence
$$
\norm{\Lambda^n(T)} = \max_{k_1,\ldots,k_m} (s_1^{k_1} \cdots s_m^{k_m}) \norm{T}_e^{k_0} = s_1 \cdots s_n.
$$
\end{proof}

We can now give a short proof of Theorem \ref{thmLambdaBoundedness}.
\begin{proof}[Proof of Theorem \ref{thmLambdaBoundedness}]
It follows immediately from Lemma \ref{lemNormExteriorPower} that 
$$
\|\bigoplus_{n=0}^N \Lambda^n(T)\| = \prod_{n=1}^N \max(1,s_n(T)).
$$
This shows the equivalence of \ref{itmBoundedness} and \ref{itmNormProduct}, and gives the formula for $\norm{\Lambda(T)}$ when these conditions hold.
Oberve that \ref{itmDecomposition} and \ref{itmNormProduct} clearly both hold when $\norm{T}_e < 1$, and clearly both fail when $\norm{T}_e > 1$, so we may asume that $\norm{T}_e = 1$.
Since all of the conditions are unchanged by replacing $T$ by $\abs{T}$, we assume that $T \ge 0$.

If \ref{itmDecomposition} holds, we have $s_n(T) \le 1 + s_n(X)$, and thus $\prod s_n(T) = \prod \max(1,s_n(T)) < \infty$, as desired.
Now assume that \ref{itmNormProduct} holds.
Let $P_{\le 1}$ be the spectral projection of $T$ corresponding to the interval $[0,1]$, and let $P_{>1} = 1 - P_{\le 1}$.
Then $T = (TP_{\le 1} + P_{>1}) + (T-1)P_{> 1}$.
Clearly $TP_{\le 1} + P_{> 1}$ is a contraction, and
$$
s_n\big((T-1)P_{>1}\big) = s_n(T) - 1
$$
is summable by assumption, so $(T-1)P_{> 1}$ is trace class.
\end{proof}

\section{Restriction operators and the Fisher-Micchelli phenomenon}\label{secRestriction}

Let $U,V \in \cJ^\infty$, and suppose $U \subseteq V$.
Define the restriction operator $R:H^2(V) \to H^2(U)$ by $Rf = f|_U$.
In this section we will explore some consequences of the following simple observation: $R = W_\iota$, where $\iota:U \hookrightarrow V$ the inclusion map, and we have chosen the positive square root of $\iota^\prime$.
When combined with the fact that $W_\varphi$ is unitary when $\varphi$ is biholomorphic, this observation is surprisingly powerful.
To begin, we note that restriction operators characterize the modulus $\abs{W_\varphi}$ for arbitrary $U,V \in \cJ^\infty$ and $\varphi \in \cJ^\infty(U,V)$.
\begin{Proposition}\label{propRModulus}
Let $U,V \in \cJ^\infty$, let $\varphi \in \cJ^\infty(U,V)$, and let $R:H^2(V) \to H^2(\varphi(U))$ be the restriction operator.
Then $\abs{W_\varphi} = \abs{R}$.
\end{Proposition}
\begin{proof}
If we write $\sigma:\varphi(U) \to U$ for the biholomorphic right inverse of $\varphi$ (i.e. $\varphi \circ \sigma = \iota$), then $W_\sigma$ is unitary and
$$
\abs{R} = \abs{W_\iota} = \abs{W_\sigma W_\varphi} = \abs{W_\varphi}.
$$
\end{proof}
By Proposition \ref{propRModulus}, we can conclude that $\abs{W_\varphi}$ is determined by the inclusion $\varphi(U) \subset V$, and not the map $\varphi$ itself.
An even more basic consequence of the fact that $R = W_\iota$ is that Theorem \ref{thmSingularValues} applies to $R$.
\begin{Corollary}\label{corRestrictionSingularValues}
Let $U,V \in \cJ^\infty$ with $U \subseteq V$, and let $R: H^2(V) \to H^2(U)$ be the restriction operator.
Assume that $\partial U \cap \partial V \ne \emptyset$.
Then $\norm{R}_e = 1$, and $s_n(R) \le 1 + \frac{K}{n}$ for some $K > 0$.
We can write $R = A + X$, with $\norm{A} \le 1$ and $X$ lying in every Schatten $p$-class with $p > 1$.
\end{Corollary}
\begin{proof}
The conclusions about $\norm{R}_e$ and $s_n(R)$ follow immediately from Theorem \ref{thmSingularValues} and the observation that $R = W_\iota$.
The decomposition $R = A + X$ is a straightforward consequence of the bound on $s_n(R)$ (see the proof of Theorem \ref{thmLambdaBoundedness}).
\end{proof}

\begin{Question}
Is there a  more direct proof of Corollary \ref{corRestrictionSingularValues} which does not use weighted composition operators?
\end{Question}

The restriction operator $R$ between Hilbert spaces with a reproducing kernel plays an important role in approximation theory and has been extensively studied, primarily in the case when $\overline{U} \subset V$.

An important feature of the Hardy spaces in question is the boundedness of the point evaluations inside the domain, or equivalently, the existence of the reproducing kernel.
Specifically, if $\{h_j\}_{j \in \N}$ is an orthonormal basis for $H^2(U)$, we recall that
$$
K_U(z,w) = \sum_{j \in \N} h_j(z) \overline{h_j(w)},
$$
is a positive definite Hermitian kernel, with the reproducing property
$$ 
\int_{\partial U} K_U(z,w) f(w) \, ds(w) = f(z), \ \ z \in U, \ \ f \in H^2(U).
$$
Due to the smoothness assumption on the boundary of $U$, the reproducing kernel is of class $C^\infty$ on the product space $\overline{U} \times \overline{U}$ minus the diagonal, with a singularity of Cauchy type along the diagonal.
This can be seen via conformal mapping, importing the behavior of the kernel associated to the unit disk $\D$. 
More precisely
$$
K_{\D}(z,w) = \frac{1}{1-z\overline{w}}.
$$

The starting point for spectral analysis of the restriction operator lies in its representation as a singular integral operator:
$$
(R^\ast R f)(z) = \int_{\partial U} K_V(z,w)f(w) \, ds(w), \ \ f \in H^2(V).
$$
The bounded positive operator $R^\ast R = \abs{R}^2$ has spectrum contained in the interval $[0, \|R\|^2]$. 
Theorem \ref{thmSingularValues} asserts that,  when $\partial U \cap \partial V$ is not empty, the essential spectrum lies in the interval $[0,1]$ and hence only countably many isolated eigenvalues, possibly converging to $1$, lie in the interval $(1,\|R\|^2]$. 
Moreover, these eigenvalues, arranged into decreasing order
$$ 
\lambda_0 \geq \lambda_1 \geq \lambda_2 \geq \ldots > 1
$$
decay as stated in the theorem:
$$ 
\sup_n n(\lambda_n -1) < \infty.
$$

To every isolated eigenvalue $\lambda_n$ we may associate an eigenfunction $f_n \in H^2(V)$ with $\| f_n \|_{H^2(V)} = 1$.
The integral equation
$$
\lambda_n f_n(z) = \int_{\partial U} K_V(z,w)f_n(w) \, ds(w), \ \ z \in V,
$$
gives some precious information about qualitative properties of $f_n$. 
For instance, the function $f_n(z)$ analytically extends across $\partial V$ as far as the reproducing kernel extends. 
Moreover, the operator $R^\ast R$ is self-adjoint, hence the double orthogonality of these
eigenfunctions:
$$ 
\langle f_n, f_k \rangle_{H^2(V)} = \delta_{kn},
$$
and
$$
\langle f_n, f_k \rangle_{H^2(U)} = \lambda_n \delta_{kn}.
$$

By Proposition \ref{propRModulus}, for every weighted composition operator $W_\varphi$ we can find a restriction operator $R$ such that $W_\varphi^*W_\varphi = R^*R$, and thus we may use restriction operators to study the operators $W_\varphi$.
The following theorem uses this relationship between $\abs{W_\varphi}$ and $\abs{R}$ to give an analog of the Fisher-Micchelli phenomenon \cite{FisherMicchelli80,Fisher83} for non-compact restriction operators, and thus for $\abs{W_\varphi}$.
\begin{Theorem}\label{thmFisherMicchelli}
Let $U,V \in \cJ^\infty$ and $\varphi \in \cJ^\infty(U,V)$, and assume that $\varphi(U) \ne V$.
Then we have
\begin{enumerate}
\item All eigenvalues of $\abs{W_\varphi}$ have multiplicity one.
\item If $\lambda_0 > \lambda_1 > \cdots > \norm{W_\varphi}_e$ is the complete list of eigenvalues of $\abs{W_\varphi}$ which exceed $\norm{W_\varphi}_e$, then non-zero eigenfunctions with eigenvalue $\lambda_n$ have at most $n$ zeros in $V$.
\end{enumerate}
In particular, if $U \subsetneq V$ then the above results hold for the restriction operator $W_\iota = R:H^2(V) \to H^2(U)$.
\end{Theorem}
\begin{proof}
By composing $\varphi$ with Riemann maps for $U$ and $V$ (which corresponds to composing $W_\varphi$ with unitary weighted composition operators) we may assume without loss of generality that $U = V = \D$.
We will reuse the symbol $U$ for the domain $\varphi(\D)$.

If $\overline{U} \subset \D$, this result is the usual Fisher-Micchelli phenomenon for compact restriction operators (see \cite[Thm. 6.2]{Fisher83}).
Thus we will assume that $\partial U \cap \partial \D \ne \emptyset$, in which case $\norm{W_\varphi}_e = 1$ by Theorem \ref{thmSingularValues}.
Now suppose that $\abs{W_\varphi}f = \abs{R}f = \lambda f$ for a non-zero function $f$.
Since $R$ is injective we must have $\lambda \ne 0$.
Then for $z \in \D$,
\begin{equation}\label{eqnRestrictionSinguarValue}
f(z) = \lambda^{-2} (R^*Rf)(z) = \lambda^{-2} \int_{\partial U} \frac{f(w)}{1-z \overline{w}} \, ds(w),
\end{equation}
where $ds$ is arclength measure on $\partial U$.

Let $S$ be the interior of $\partial \D \setminus \partial U$, which is non-empty by assumption.
From \eqref{eqnRestrictionSinguarValue} we can see that $f$ extends holomorphically to a neighborhood of $S$.
Arguing as in \cite[Thm. 6.2]{Fisher83}, one can show that for $z \in S$ we have
$$
\lambda^2 \abs{f(z)}^2 = \int_{\partial U} \abs{f(w)}^2 \Re \left( \frac{z+w}{z-w} \right) \, ds(w) > 0,
$$
where we used that $\Re \left( \frac{z+w}{z-w} \right)$ is nonnegative on $\partial U$ and non-vanishing in $\partial U \cap \D$.
Since $f$ cannot vanish at any point of $S$, as an immediate consequence we have that the eigenvalue $\lambda^2$ of $R^* R$ has multiplicity one.

We will study the eigenfunctions of $\abs{W_\varphi} = \abs{R}$ by approximating $U$ by domains $U_t$ which are compactly contained in $\D$.
Specifically, for $t \in [0,1]$ let $\varphi_t:\D \to \D$ be the map $\varphi_t(z) = \varphi(tz)$.
Let $U_t = \varphi_t(\D)$ and let $R_t:H^2(\D) \to H^2(U_t)$ be the restriction operator.
By Proposition \ref{propRModulus} we have $\abs{R_t} = \abs{W_{\varphi_t}}$.

For $t \in [0,1]$ let $\rho_t:\D \to \D$ be given by $\rho_t(z) = tz$.
Hence $W_{\varphi_t} = W_{\rho_t} W_{\varphi}$.
Since $W_{\rho_t} \to 1$ in the strong operator topology as $t \uparrow 1$, we have $W_{\varphi_t} \to W_\varphi$ strongly as well.
Moreover, if $s \ge t > 0$ we have $W_{\rho_{t/s}}W_{\varphi_s} = W_{\varphi_t} $.
Since $0 \le W_{\rho_{t/s}} \le 1$, we have $\norm{W_{\varphi_s}g} \ge \norm{W_{\varphi_t}g}$ for all $g \in H^2(\D)$.
Hence $\norm{W_{\varphi_t}} \to \norm{W_\varphi}$ as $t \uparrow 1$.

We can now repeat this argument with the exterior powers $\Lambda^n (W_{\varphi_t})$ to obtain convergence of the smaller singular values as well.
Since $W_{\varphi_t} \to W_{\varphi}$ strongly and $\norm{W_{\varphi_t}}$ is bounded, we have $\Lambda^n(W_{\varphi_t}) \to \Lambda^n(W_{\varphi})$ strongly.
When $1 \ge s \ge t > 0$ we have $\Lambda^n(W_{\rho_{t/s}}) \Lambda^n(W_{\varphi_s}) = \Lambda^n(W_{\varphi_t})$ and $\Lambda^n(W_{\rho_{t/s}}) \le 1$, so that $\norm{\Lambda^n(W_{\varphi_t})} \to \norm{\Lambda^n (W_\varphi)}$ as $t \uparrow 1$.
By Lemma \ref{lemNormExteriorPower}, this means that 
$$
\prod_{j=1}^n s_j(W_{\varphi_t}) \to \prod_{j=1}^n s_j(W_{\varphi}).
$$
Since $s_j(W_\varphi) \ge 1$ for all $j$, it follows that $s_n(W_{\varphi_t}) \to s_n(W_\varphi)$ for all $n \ge 1$.

Now suppose $\lambda_n = s_{n-1}(W_\varphi)$ is an eigenvalue of $\abs{W_\varphi}$ and $\lambda_n > 1$.
Let $\lambda_{0,t} > \lambda_{1,t} > \cdots$ be the sequence of eigenvalues of $\abs{W_{\varphi_t}}$.
Note that for all $t \in [0,1]$ and $k \ge n$, the eigenvalue $\lambda_{k,t}$ has multiplicity one and the function $t \mapsto s_k(W_{\varphi_t})$ is continuous in $t$. 
Thus we can choose a small neighborhood $(\lambda_n - \epsilon, \lambda_n + \epsilon)$ such that 
$$
\sigma(\abs{W_{\varphi_t}}) \cap (\lambda_n - 2\epsilon, \lambda_n + 2\epsilon) = \{\lambda_{n,t}\}
$$
for $t$ sufficiently close to $1$.
For such $t$, the rank-one projection $P_{n,t}$ onto the eigenspace $\ker (\lambda_{n,t} - \abs{W_{\varphi_t}})$ is given by
$$
P_{n,t} = \int_{\abs{z-\lambda_n}=\epsilon} (z - \abs{W_{\varphi_t}})^{-1} \, dz.
$$

Since $W_{\varphi_t}^* = W_{\varphi}^* W_{\rho_t}$, we have $W_{\varphi_t}^* \to W_{\varphi}^*$ as $t \uparrow 1$, and since multiplication is jointly strongly continuous on bounded sets, we have $W_{\varphi_t}^*W_{\varphi_t} \to W_\varphi^*W_\varphi$ strongly as $t \uparrow 1$, and thus $\abs{W_{\varphi_t}} \to \abs{W_\varphi}$.
We may exchange this limit with the integral giving $P_{n,t}$ to obtain $P_{n,t} \to P_{n,1}$ strongly as $t \uparrow 1$.
Thus we may choose unit norm eigenfunctions $f_{n,t}$ corresponding to the eigenvalue $\lambda_{n,t}$ for $\abs{W_{\varphi,t}}$ such that $f_{n,t} \to f_{n,1}$ as $t \uparrow 1$.
Since $\abs{W_{\varphi_t}} = \abs{R_t}$, by the Fisher-Micchelli theorem each $f_{n,t}$ has exactly $n$ zeros in $\D$.
Thus an application of the argument principle yields that $f_{n,1}$ has at most $n$ zeroes in $\D$.
\end{proof}

\bibliographystyle{alpha}
\bibliography{wcobib} 

\begin{thebibliography}{LLQRP16}

\bibitem[Bel92]{Bell92}
Steven~R. Bell.
\newblock {\em The {C}auchy transform, potential theory, and conformal
  mapping}.
\newblock Studies in Advanced Mathematics. CRC Press, Boca Raton, FL, 1992.

\bibitem[FH05]{FewsterHollands05}
Christopher~J. Fewster and Stefan Hollands.
\newblock Quantum energy inequalities in two-dimensional conformal field
  theory.
\newblock {\em Rev. Math. Phys.}, 17(5):577--612, 2005.

\bibitem[Fis83]{Fisher83}
Stephen~D. Fisher.
\newblock {\em Function theory on planar domains}.
\newblock Pure and Applied Mathematics (New York). John Wiley \& Sons, Inc.,
  New York, 1983.
\newblock A second course in complex analysis, A Wiley-Interscience
  Publication.

\bibitem[FM80]{FisherMicchelli80}
S.~D. Fisher and Charles~A. Micchelli.
\newblock The {$n$}-width of sets of analytic functions.
\newblock {\em Duke Math. J.}, 47(4):789--801, 1980.

\bibitem[LLQRP16]{LLQR16}
Gandalf Lechner, Daniel Li, Herv{\'e} Queff{\'e}lec, and Luis
  Rodr{\'i}guez-{P}iazza.
\newblock Approximation numbers of weighted composition operators.
\newblock {\em arXiv:1612.01177 [math.FA]}, 2016.

\bibitem[Seg04]{SegalDef}
Graeme Segal.
\newblock The definition of conformal field theory.
\newblock In {\em Topology, geometry and quantum field theory}, volume 308 of
  {\em London Math. Soc. Lecture Note Ser.}, pages 421--577. Cambridge Univ.
  Press, Cambridge, 2004.

\bibitem[Sha87]{Shapiro87}
Joel~H. Shapiro.
\newblock The essential norm of a composition operator.
\newblock {\em Ann. of Math. (2)}, 125(2):375--404, 1987.

\bibitem[Sha93]{Shapiro93}
Joel~H. Shapiro.
\newblock {\em Composition operators and classical function theory}.
\newblock Universitext: Tracts in Mathematics. Springer-Verlag, New York, 1993.

\bibitem[Ten16a]{Ten16a}
James~E. Tener.
\newblock Construction of the unitary free fermion {S}egal {CFT}.
\newblock {\em arXiv:1608.02095 [math-ph]}, 2016.

\bibitem[Ten16b]{Ten16b}
James~E. Tener.
\newblock Geometric realization of algebraic conformal field theories.
\newblock {\em arXiv:1611.01176 [math-ph]}, 2016.

\end{thebibliography}

\end{document}